\documentclass[11pt,english]{amsart}

\usepackage[mathscr]{eucal}
\usepackage{amsmath,amssymb,amsfonts,amsthm,enumerate}

\usepackage[textwidth=16cm,
            textheight=21cm,
            lmargin=2.7cm,
            rmargin=2.7cm,
            marginparsep=0.1cm,
            marginparwidth=2.4cm,
            footskip=1cm]
           {geometry}

\newtheorem{theorem}{Theorem}[section]
\newtheorem{lemma}[theorem]{Lemma}
\newtheorem{corollary}[theorem]{Corollary}

\newtheorem{conjecture}[theorem]{Conjecture}

\newtheorem{remark}[theorem]{Remark}

\newcommand{\Z}{\mathbb Z}

\newcommand{\C}{\mathbb C}

\newcommand{\be}{\begin{equation}}
\newcommand{\ee}{\end{equation}}
\newcommand{\und}{\;\mbox{ and }\;}
\newcommand{\nn}{\nonumber}
\newcommand{\ber}{\begin{eqnarray}}
\newcommand{\eer}{\end{eqnarray}}

\newcommand{\Summ}[1]{\underset{#1}{\sum}}

\newcommand\floor[1]{{\lfloor{#1}\rfloor}}

\newcommand{\mb}{\mathbb}

\author{Pablo Candela}
\address{Universidad Aut\'onoma de Madrid, and ICMAT\\ Madrid 28049\\ Spain}
\email{pablo.candela@uam.es}

\author{Diego Gonz\'alez-S\'anchez}
\address{Universidad Aut\'onoma de Madrid, and ICMAT\\ Madrid 28049\\ Spain}
\email{diego.gonzalezs@predoc.uam.es}

\author{David J. Grynkiewicz}
\address{University of Memphis\\
Department of Mathematical Sciences\\
Memphis, TN 38152\\
USA}
\email{diambri@hotmail.com}

\title{On sets with small sumset and $m$-sum-free sets in $\Z/p\Z$}

\begin{document}

\begin{abstract}
The $3k-4$ conjecture in groups $\Z/p\Z$ for $p$ prime states that if $A$ is a nonempty subset of $\Z/p\Z$ satisfying $2A\neq \Z/p\Z$ and $|2A|=2|A|+r \leq \min\{3|A|-4,\;p-r-4\}$, then $A$ is covered by an arithmetic progression of size at most $|A|+r+1$. A theorem of Serra and Z\'emor proves the conjecture provided $r\leq 0.0001|A|$, without any additional constraint on $|A|$. Subject to the mild additional constraint $|2A|\leq 3p/4$ (which is optimal in a sense explained in the paper), our first main result improves the bound on $r$, allowing $r\leq 0.1368|A|$. We also prove a variant which further improves this bound on $r$ provided $A$ is sufficiently dense. We then give several applications. First we apply the above variant to give a new upper bound for the maximal density of $m$-sum-free sets in $\Z/p\Z$, i.e., sets $A$ having no solution $(x,y,z)\in A^3$ to the equation $x+y=mz$, where $m\geq 3$ is a fixed integer. The previous best upper bound for this maximal density was $1/3.0001$ (using the Serra-Z\'emor Theorem). We improve this to $1/3.1955$. We also present a construction following an idea of Schoen, which yields a lower bound for this maximal density of the form $1/8+o(1)_{p\to\infty}$. Another  application of our main results concerns sets of the form $\frac{A+A}{A}$ in $\mb{F}_p$, and we also improve the structural description of large sum-free sets in $\mb{Z}/p\Z$.
\end{abstract}

\subjclass[2010]{11P70, 11B13, 05B10}


\thanks{This work has benefited from support from the Spanish Ministerio de Ciencia e Innovaci\'on project MTM2017-83496-P, and from the La Caixa Foundation (ID 100010434) under agreement LCF/BQ/SO16/52270027}
\maketitle

\section{Introduction}
Given a subset $A$ of an abelian group $G$, we often denote the sumset $A+A=\{x+y:x,y\in A\}$ by $2A$, and we denote the complement $G\setminus A$ by $\overline{A}$.

One of the central topics in additive number theory is the study of the structure of a finite subset $A$ of an abelian group under the assumption that the sumset $2A$ is small. In this paper, we focus on groups $\Z/p\Z$ of integers modulo a prime $p$, and on the regime in which the \emph{doubling constant} $|2A|/|A|$ is within a small additive constant of the minimum possible value.

To put this in context, let us recall the basic fact that a finite set $A$ of integers always satisfies $|2A|\geq 2|A|-1$ and that this minimum is attained only if $A$ is an arithmetic progression (see \cite[Theorem 3.1]{Gr}). This description of extremal sets is extended by a result of Freiman, known as the $3k-4$ Theorem, which tells us that $A$ is still efficiently covered by an arithmetic progression even when $|2A|$ is as large as $3|A|-4$.

\begin{theorem}[Freiman's $3k-4$ Theorem]\label{thm-3k4Z-sym}
Let $A\subseteq \Z$ be a finite set satisfying $|2A|\leq 3|A|-4$. Then there is an arithmetic progression $P\subseteq \Z$ such that $A\subseteq P$ and $|P|\leq |2A|-|A|+1$.
\end{theorem}
For sets $A$ in $\Z/p\Z$ with $2A\neq \Z/p\Z$, the Cauchy-Davenport Theorem \cite[Theorem 6.2]{Gr} gives the lower bound analogous to the one for $\Z$ mentioned above, namely $|2A|\geq 2|A|-1$, and the description of extremal sets as arithmetic progressions (when $|2A|< p-1$) is given by Vosper's Theorem  \cite[Theorem 8.1]{Gr}.

It is widely believed that an analogue of Freiman's $3k-4$ Theorem holds for subsets of $\Z/p\Z$ under some mild additional upper bound on $|2A|$ (or on  $|A|$). More precisely, the following conjecture is believed to be true (see \cite[Conjecture 19.2]{Gr}), describing efficiently not just $A$, but also $2A$, in terms of progressions.

\begin{conjecture}\label{conj-3k-4}
Let $p$ be a prime and let $A\subset \Z/p\Z$ be a nonempty subset satisfying $2A\neq \Z/p\Z$ and $|2A|=2|A|+r \leq \min\{3|A|-4,\;p-r-4\}$. Then there exist arithmetic progressions $P_A,P_{2A}\subseteq \Z/p\Z$ with the same difference such that $A\subseteq P_A$, $|P_A|\leq |A|+r+1$, $P_{2A}\subseteq 2A$, and $|P_{2A}|\geq 2|A|-1$.
\end{conjecture}

Progress toward this conjecture was initiated by Freiman himself, who proved in \cite{Freiman61} that the conclusion concerning $P_A$ holds provided that  $|2A| \leq 2.4|A| - 3$ and $|A| < p/35$. Since then, there has been much work  improving Freiman's result in various ways. For instance, R{\o}dseth showed in \cite{Rodseth06} that the constraint $|A| < p/35$  can be weakened to $|A| < p/10.7$ while maintaining the doubling constant $2.4$. In \cite{GR06}, Green and Ruzsa pushed the doubling constant up to 3, at the cost of a stronger constraint $|A| < p/10^{215}$. In \cite{SZ09}, Serra and Z{\'e}mor obtained a result with no constraint on $|A|$ other than the bounds on $|2A|$ in the conjecture, with the same conclusion concerning $P_A$, but at the cost of reducing the doubling constant, namely, assuming that $|2A| \leq (2+\alpha)|A|$ with $\alpha<0.0001$. See also \cite{CSS,LS} for recent improvements on the doubling constant $2.4$ in Freiman's result. The book \cite{Gr} presents various other results towards Conjecture \ref{conj-3k-4}, in a treatment covering many of the methods from the works mentioned above.

In this paper, we establish the following new result regarding Conjecture \ref{conj-3k-4}, which noticeably improves the doubling constant obtained by Serra and Z\'emor in \cite{SZ09} at the cost of only adding the constraint $|2A|\leq \frac34 p$.

\begin{theorem}\label{thm:main1}
Let $p$ be prime, let $A\subseteq \Z/p\Z$ be a nonempty subset with $|2A|=2|A|+r$, and  let $\alpha\approx 0.136861$ be the unique real root of the cubic $4x^3+9x^2+6x-1$. Suppose \begin{align*}  &|2A|\leq (2+\alpha)|A|-3\quad\und\quad |2A|\leq\frac34 p.\end{align*}
Then there exist arithmetic progressions $P_A,P_{2A}\subseteq \Z/p\Z$ with the same difference such that $A\subseteq P_A$, $|P_A|\leq |A|+r+1$, $P_{2A}\subseteq 2A$, and $|P_{2A}|\geq 2|A|-1$.
\end{theorem}
Unlike in \cite{SZ09}, here we do have a constraint on $|A|$ in the form of the upper bound $|2A|\leq \frac34 p$. However, this upper bound is still optimal in the following weak sense. The conjectured upper bound on $|2A|$ (given by Conjecture \ref{conj-3k-4}) is $p-r-4$. However, in the extremal case where $r=|A|-4$ (the largest value of $r$ allowed in Conjecture \ref{conj-3k-4}), the conjectured bound implies $3|A|-4=|2A|\leq p-|A|$, whence $|A|\leq \frac{p+4}{4}$ and $|2A|=3|A|-4\leq \frac{3p}{4}-1$.
Thus, the  bound $p-r-4$ becomes as small  as $\frac{3p}{4}-1$ as we \emph{range} over all allowed values for $\alpha$ and $|A|$, making $\frac34p$ the optimal bound independent of $\alpha$ and $r$.

Let us emphasize that our improvement upon the Serra--Z\'emor result (i.e.\ our weakening of the constraint on $\alpha$) is valid for $|A| \leq \frac{0.75 p +3}{2+\alpha}$, whereas the natural upper bound on $|A|$ given by Conjecture \ref{conj-3k-4} is larger, namely $|A|\leq \frac{p+2}{2+2\alpha}$. Therefore, in the regime $\frac{0.75 p +3}{2+\alpha}<|A|\leq \frac{p+2}{2+2\alpha}$, our result does not improve on that of Serra and Z\'emor.

We also prove the following variant of Theorem \ref{thm:main1}, which is optimized for sets $A$ whose  density is large but at most 1/3. This optimization is designed for an application concerning $m$-sum-free sets, which we discuss below.

\begin{theorem}\label{thm:main2}
Let $p$ be prime, let $\eta\in (0,1)$, let $A\subseteq \Z/p\Z$ be a set with $|A|\geq \eta\, p>0$ and $|2A|=2|A|+r< p$, and let $\alpha= -\frac{5}{4} + \frac{1}{4}\sqrt{9+8\,\eta\, p\sin(\pi/p)/\sin(\pi\eta/3)}$. Suppose
\begin{align*}
&|2A|\leq (2+\alpha)|A|-3\quad\und\quad |A|\leq\frac{p-r}{3}.
\end{align*}
Then there exist arithmetic progressions $P_A,P_{2A}\subseteq \Z/p\Z$ with the same difference  such that $A\subseteq P_A$, $|P_A|\leq |A|+r+1$, $P_{2A}\subseteq 2A$, and  $|P_{2A}|\geq 2|A|-1$.
\end{theorem}
We apply this result to obtain new upper bounds for the size of $m$-sum-free sets in $\Z/p\Z$. For a positive integer $m$, a subset $A$ of an abelian group is said to be \emph{$m$-sum-free} if there is no triple $(x,y,z)\in A^3$ satisfying $x+y=mz$. These sets have been studied in  numerous works in arithmetic combinatorics, including various types of abelian group settings \cite{B&al,C&G,C&G-cts,D&P,M&R} (see also \cite[Section 3]{C&R} for an overview of this topic). In $\Z/p\Z$, a central goal concerning these sets is to estimate the quantity
\be\label{eq:def-d_m}
d_m(\Z/p\Z)=\max\big\{\tfrac{|A|}{p}:A\subseteq \Z/p\Z\textrm{ $m$-sum-free}\big\}.
\ee
This goal splits naturally into two problems of different nature. On one hand, we have the case $m=2$, which is the only one in which the solutions of the linear equation in question (i.e., 3-term arithmetic progressions) form a translation invariant set. Roth's Theorem \cite{Roth} tells us that $d_2(\Z/p\Z)\to 0$ as $p\to\infty$, and the problem in this case is then the well-known one of determining the optimal bounds for Roth's theorem, i.e., how fast $d_2(\Z/p\Z)$ vanishes as $p$ increases (recent developments in this direction include \cite{Bloom,Sanders}). On the other hand, we have the cases $m\geq 3$.  For each of these, the above-mentioned translation-invariance fails, and it is known that $d_m(\Z/p\Z)$ converges, as $p\to\infty$ through primes, to a positive constant $d_m$ which can be modeled on the circle group (see \cite{C&S}), the problem then being to determine this constant. Our application of Theorem \ref{thm:main2} makes progress on the latter problem.

Note that, if $A$ is $m$-sum-free, then the \emph{dilate}
$m\cdot A=\{mx:\;x\in A\}\subseteq \Z/p\Z$ satisfies $2A\cap m\cdot A =\emptyset$, whence, if $m$ and $p$ are coprime, we have $|2A|+|m\cdot A|=|2A|+|A|\leq p$. Combining this with the bound $|2A|\geq 2|A|-1$ given by the Cauchy-Davenport Theorem, we deduce the simple bound $|A|\leq \frac{p+1}{3}$, which implies in particular that $d_m\leq 1/3$. It was noted in \cite{C&R} that partial versions of Conjecture \ref{conj-3k-4} can be used to improve on this bound, provided these versions are applicable to sets of density up to $1/3$. The best version available for that purpose in \cite{C&R} was given by the theorem of Serra and Z\'emor mentioned above, and this resulted in the first upper bound for $d_m$ below $1/3$, namely $1/3.0001$ (see \cite[Theorem 3.1]{C&R}). In this paper, using Theorem \ref{thm:main2} we obtain the following improvement.

\begin{theorem}\label{thm:dm-ub1}
Let $p\geq 80$ be a prime, let $m$ be an integer in $[2,p-2]$, and let $c=c(p)$ be the solution to the equation $\Big(7+\sqrt{8\,c\, p\sin(\pi/p)/\sin(\pi c/3)+9}\Big)\, c= 4+\frac{12}{p}$. Then $d_m(\Z/p\Z)<c$. In particular, $d_m\leq \frac{1}{3.1955}$.
\end{theorem}

The following observation, relating this theorem to the study of sum-products in the field $\mb{F}_p$, was made by the anonymous referee: if $(A+A)\cap m\cdot A$ contains a non-zero element and $0\notin A$, then $m$ is in the set $\frac{A+A}{A}:=\{ (a_1+a_2)a_3^{-1}:a_1, a_2, a_3\in A\}\subset \mb{F}_p$, and therefore Theorem \ref{thm:dm-ub1} has the following consequence.

\begin{corollary}
If $A\subset \mb{F}_p\setminus\{0\}$ satisfies $|A|\geq 0.313\,p$, then for $p$ sufficiently large we have $ \mb{F}_p\setminus\{-1,0,1\}\subseteq \frac{A+A}{A}$.
\end{corollary}

\noindent This result is an analogue, for sets $\frac{A+A}{A}$, of Theorem 1.1 in \cite{BHS}, which says that if $A\subset \mb{F}_p$ has $|A|\geq 0.3051 \,p$, then for $p$ sufficiently large we have $\mb{F}_p\setminus\{0\}\subseteq (A+A)A:=\{ (a_1+a_2)a_3:a_i\in A\}$.

Regarding lower bounds for $d_m(\Z/p\Z)$, note that, identifying $\Z/p\Z$ with the integers $[0,p-1]$, the interval $(\frac{2}{m^2-4} p,\frac{m}{m^2-4}p)$ is an $m$-sum-free set. This set has asymptotic  density $\frac{1}{m+2}$, and is still the greatest known example for $m\leq 6$. However, for larger values of $m$, a construction of Tomasz Schoen (personal communication), presented in this paper in Lemma \ref{Schoen} in an optimized form thanks to indications of the anonymous referee, yields the improved lower bound $d_m\geq \frac18$. The following theorem summarizes these results.
\begin{theorem}\label{thm:lbs}
For $m\leq 6$, we have $d_m\geq \frac{1}{m+2}$. For $m\geq 7$, we have $d_m\geq \frac{1}{8}$.
\end{theorem}
Our final application concerns the study of large sum-free sets in $\Z/p\Z$ (i.e.\ the case $m=1$ of $m$-sum-free sets as defined above). It is well-known, by the argument using the Cauchy-Davenport Theorem mentioned above, that a sum-free set in $\Z/p\Z$ has size at most $\lfloor (p+1)/3 \rfloor$, and that this bound is attained by the interval $I=(p/3,2p/3)\subset \Z/p\Z$ and by any non-zero dilate of $I$. Several works have studied the question of the robustness of this structural description, namely, whether every sum-free set in $\Z/p\Z$ of density close to $1/3$ must resemble a dilate of $I$. In this direction the following theorem was proved by Deshouillers and Lev in  \cite{D&L}.
\begin{theorem}\label{thm:D&L}
Let $p$ be a sufficiently large prime, and suppose that $A\subset \Z/p\Z$ is sum-free. If $|A| > 0.318\, p$, then there exists $d\in\mb{Z}$ such that $A \subset d \cdot [\, |A| , p - |A|\,]$.
\end{theorem}
\noindent Applying Theorem \ref{thm:main2}, we improve the constant $0.318$ to $0.313$.

The paper is laid out as follows. In Section \ref{sec:3k-4}, we prove Theorems \ref{thm:main1} and \ref{thm:main2}. Our results on $m$-sum-free sets are proved in Section \ref{sec:msf}. There, in Subsection \ref{subsec:lbs}, we present the above construction and deduce Theorem \ref{thm:lbs}. In Subsection \ref{subsec:ubs}, we apply Theorem \ref{thm:main2} to obtain Theorem \ref{thm:dm-ub1}. Finally, in Subsection \ref{subsec:sumfree} we obtain the above-mentioned improvement of Theorem \ref{thm:D&L}.

\section{New bounds toward the $3k-4$ conjecture in $\Z/p\Z$}\label{sec:3k-4}

Our first task in this section is to prove Theorem \ref{thm:main1}. We shall obtain this result as the special case $\varepsilon=3/4$ of the following theorem.

\begin{theorem}\label{fouier-nightmare-calc} Let $p$ be prime, let $0<\varepsilon\leq \frac34$ be a real number,  let $\alpha$ be the unique positive root of the cubic $4x^3+(12-4\varepsilon)x^2+(9-4\varepsilon)x+(8\varepsilon -7)$, and  let $A\subseteq \Z/p\Z$ be a nonempty subset with $|2A|=2|A|+r$. Suppose \begin{align*}  &|2A|\leq (2+\alpha)|A|-3\quad\und\quad |2A|\leq\varepsilon\, p.\end{align*}
Then there exist arithmetic progressions $P_A,P_{2A}\subseteq \Z/p\Z$ with the same difference  such that $A\subseteq P_A$, $|P_A|\leq |A|+r+1$,  $P_{2A}\subseteq 2A$, and $|P_{2A}|\geq 2|A|-1$.
\end{theorem}

The proof is a modification of the argument used to prove \cite[Theorem 19.3]{Gr}, itself based on the original work of Freiman \cite{Freiman61} and incorporating  improvements to the calculations noted by R{\o}dseth \cite{Rodseth06}. The main new contribution is an argument to allow the restriction $|2A|\leq \frac12(p+3)$ from \cite[Theorem 19.3]{Gr} to be replaced by the above condition $|2A|\leq \varepsilon p$. For $\varepsilon=3/4$, this is optimal in the sense explained in the introduction.

In the proof of Theorem \ref{fouier-nightmare-calc}, we use the following version of the $3k-4$ Theorem for $\Z$. Here, for $X\subseteq \Z$, we denote the greatest common divisor $\gcd(X-X)$ by $\gcd^*(X)$. Note, for $|X|\geq 2$, that $d=\gcd^*(X)$ is the minimal $d\geq 1$ such that $X$ is contained in an arithmetic progression with difference $d$. We remark that, when $B=-A$, we have $P_{A-A}\subseteq A-A$ and $-P_{A-A}\subseteq -(A-A)=A-A$. Since $2|P_{A-A}|\geq 4|A|-2> |A-A|$, the progressions $P_{A-A}$ and $-P_{A-A}$ intersect, ensuring that $P=P_{A-A}\cup -P_{A-A}\subseteq A-A$  is a progression contained in $A-A$ with $|P|\geq 2|A|-1$ and $-P=P$. Thus the progression $P_{A-A}$ in Theorem \ref{thm-3k4Z} can be assumed to be symmetric (i.e., centered at the origin) when $B=-A$.

\begin{theorem}\label{thm-3k4Z}
Let $A,\,B\subseteq \Z$ be finite, nonempty subsets with $\gcd^*(A+B)=1$  and $$|A+B|=|A|+|B|+r\leq |A|+|B|+\min\{|A|,\,|B|\}-3-\delta,$$ where $\delta=1$ if  $x+A=B$ for some $x\in \Z$, and otherwise $\delta=0$.
Then there are arithmetic progressions $P_A,\,P_B,\,P_{A+B}\subseteq \Z$ with common difference $1$ such that $A\subseteq P_A$, $B\subseteq P_B$, $P_{A+B}\subseteq A+B$, $|P_A|\leq |A|+r+1$, $|P_B|\leq |B|+r+1$ and $|P_{A+B}|\geq |A|+|B|-1$.
\end{theorem}

Let $G$ and $G'$ be abelian groups and let $A,\,B\subseteq G$. A \emph{Freiman isomorphism} is a well-defined map $\psi:A+B\rightarrow G'$ defined by two coordinate maps $\psi_A:A\rightarrow G'$ and $\psi_B:B\rightarrow G'$ such that $\psi(x+y)=\psi_A(x)+\psi_B(y)$ for all $x\in A$ and $y\in B$. That $\psi$ is well-defined is equivalent to the statement that $\psi_A(x_1)+\psi_B(y_1)=\psi_A(x_2)+\psi_B(y_2)$ whenever $x_1+y_1=x_2+y_2$, for $x_1,\,x_2\in A$ and $y_1,\,y_2\in B$, and $\psi_A(A)+\psi_B(B)$ is then the homomorphic image of $A+B$. It is an isomorphism if $\psi$ is injective on $A+B$, which is equivalent to $\psi_A(x_1)+\psi_B(y_1)=\psi_A(x_2)+\psi_B(y_2)$ holding if and only if $x_1+y_1=x_2+y_2$, for $x_1,\,x_2\in A$ and $y_1,\,y_2\in B$. We denote this by $A+B\cong \psi_A(A)+\psi_B(B)$. A Freiman homomorphism  $\psi:A+B\rightarrow G'$ on the sumset defines a Freiman homomorphism $\psi':A-B\rightarrow G'$ on the difference set given by $\psi'(x-y)=\psi_A(x)-\psi_B(y)$, for $x\in A$ and $y\in B$, which is an isomorphism when $\psi$ is. In the special case when $A=B$, we find that $\psi_A(x)+\psi_B(y)=\psi_A(y)+\psi_B(x)$ for all $x,\,y\in A=B$, implying $\psi_B(x)=\psi_A(x)+(\psi_B(y)-\psi_A(y))$ for any $x,\,y\in A=B$. Fixing $y\in A$ and letting $x$ range over all possible $x\in A$ shows that the map $\psi_B$ is simply a translate of the map $\psi_A$. This means it can (and generally will) be assumed  that $\psi_A=\psi_B$  for a Freiman homomorphism $\psi$ when $A=B$.
 See \cite[Chapter 20]{Gr} for a fuller discussion regarding Freiman homomorphisms.

For a prime $p$,  nonzero $g\in \Z/p\Z$ (which is then a generator of $\Z/p\Z$), and integers $m\leq n$, let $$[m,n]_g=\{mg,(m+1)g,\ldots,ng\}$$ denote the corresponding interval in $\Z/p\Z$. If $m>n$, then $[m,n]_g=\emptyset$. We define (for each $g\in \Z/p\Z\setminus\{0\}$) a function $\ell_g$ from the set of subsets $X\subset \Z/p\Z$ to $\mb{Z}_{\geq 0}$, by
\[
\ell_g(X):= \min\{|P|: P\textrm{ is an arithmetic progression of difference $g$ with }X\subset P\}.
\]
We let $\overline{X}=(\Z/p\Z)\setminus X$ denote the complement of $X$ in $\Z/p\Z$.
We say that a sumset $A+B\subseteq \Z/p\Z$ \emph{is rectifiable} if $\ell_g(A)+\ell_g(B)\leq p+1$ for some nonzero $g\in \Z/p\Z$. In such case, $A\subseteq a_0+[0,m]_g$ and $B\subseteq b_0+[0,n]_g$ with $m+n=\ell_g(A)+\ell_g(B)-2\leq p-1$, for some $a_0,\,b_0\in \Z/p\Z$, in which case the maps $a_0+sg\mapsto s$ and $b_0+tg\mapsto t$, for $s,\,t\in \Z$, when restricted to $A$ and $B$ respectively, show that the sumset $A+B$ is Freiman isomorphic (see \cite[Section 2.8]{Gr}) to an integer sumset. This allows us to canonically apply results from $\Z$ to the sumset $A+B$.

If $G$ is an abelian group and $A,\,B\subseteq G$ are  subsets, then we say that $A$ is \emph{saturated} with respect to $B$ if $(A\cup \{x\})+B\neq A+B$ for all $x\in \overline A$. In the proof of Theorem \ref{fouier-nightmare-calc}, we shall also use the following basic result regarding saturation \cite[Lemma 7.2]{Gr}, whose earlier form dates back to Vosper \cite{vosper}. We include the short proof for completeness.

\begin{lemma}\label{lemma-dual}
Let $G$ be an abelian group and let $A,\,B\subseteq G$ be subsets. Then $$-B+\overline{A+B}\subseteq \overline A$$ with equality holding if and only if $A$ is saturated with respect to $B$.
\end{lemma}

\begin{proof}
First observe that $-B+\overline{A+B}\subseteq \overline A$, for if $b\in B$, $z\in \overline{A+B}$ and by contradiction $-b+z=a$ for some $a\in A$, then $z=a+b\in A+B$, contrary to its definition.
If $A$ is saturated with respect to $B$, then given any $x\in \overline A$, there exists some $b\in B$ and $z\in \overline{A+B}$ with $x+b=z$, whence $x=-b+z\in -B+\overline{A+B}$. This shows that $\overline{A}\subseteq -B+\overline{A+B}$, and as the reverse inclusion always holds (as just shown), it follows that $\overline A=-B+\overline{A+B}$. Conversely, if $\overline A=-B+\overline{A+B}$, then given any $x\in \overline A$, there exists some $b\in B$ and $z\in \overline{A+B}$ with $x=-b+z$, implying $x+b=z\notin A+B$. Since $x\in \overline A$ is arbitrary, this shows that $A$ is saturated with respect to $B$.
\end{proof}

\begin{proof}[Proof of Theorem \ref{fouier-nightmare-calc}]

Let $f(x)=4x^3+(12-4\varepsilon)x^2+(9-4\varepsilon)x+(8\varepsilon -7)$, so that $f'(x)=12x^2+(24-8\varepsilon)x+(9-4\varepsilon)$.  Then  $f'(x)>0$  for $x\geq 0$ (in view of $\varepsilon\leq 3/4$), meaning $f(x)$ is an increasing function for $x\geq 0$ with $f(0)=8\varepsilon-7<0$ and $f(1/2)=1+5\varepsilon>0$. Consequently, $f(x)$ has a unique positive root $0<\alpha<\frac12$.

Since $|2A|\leq\varepsilon p<p$, the Cauchy-Davenport Theorem implies $r\geq -1$.
Let \be\label{oneone}\beta=\frac{r+3}{|A|}>0,\ee so that \be \label{onetwo} r=\beta|A|-3,\quad |2A|=2|A|+r=(2+\beta)|A|-3\quad\und\quad\beta\leq \alpha<\frac12.\ee

Since $2|A|+r=|2A|\leq \varepsilon p\leq\frac34 p$, it follows that $|A|\leq \frac38p-\frac12 r$, and since $r\geq -1$ we deduce that
\be\label{betabound}
|A|\leq \frac{3p+4}{8}.
\ee

The proof naturally breaks into two parts:  a first case where there is a large rectifiable sub-sumset, and a second case where there is not. The latter case will lead to a contradiction.

\medskip\noindent\textbf{Case 1: } Suppose there exist subsets $A'\subseteq A$ and $B'\subseteq A$ with $|B'|\leq |A'|$  and \be\label{rectify-large-bound-sym} |A'|+2|B'|-4\geq |2A|\ee
such that  $A'+B'$ is rectifiable. Furthermore,  choose a pair of subsets  $A'\subseteq A$ and $B'\subseteq A$ with these properties such that  $|A'|+|B'|$ is maximal, and for these subsets $A'$ and $B'$, let $g\in \Z/p\Z$ be a nonzero difference with $\ell_g(A')+\ell_g(B')\leq p+1$  minimal. Note $|A'|\geq |B'|\geq 2$; indeed, if  $|B'|\leq 1$, then combining this with the hypotheses $|B'|\leq |A'|\leq |A|$ and \eqref{rectify-large-bound-sym} yields the contradiction $|A|-2\geq |2A|\geq |A|$.
  Since $A'+B'$ is rectifiable, the Cauchy-Davenport Theorem for $\Z$ \cite[Theorem 3.1]{Gr} ensures  $$|A'+B'|=|A'|+|B'|+r'\quad\mbox{ with $r'\geq -1$}.$$  Moreover, we have \be\label{ABC-intervals-sym}A'\subseteq P_A:=a_0+[0,m]_g,\quad B'\subseteq P_B:=b_0+[0,n]_g,\quad\und\quad A'+B'\subseteq a_0+b_0+[0,m+n]_g\ee
 with $a_0,\,a_0+mg\in A'$, \ $b_0,\,b_0+ng\in B'$ and  $m+n\leq p-1$, for some $a_0,\,b_0\in \Z/p\Z$. Then, since $A'+B'$ is rectifiable, it follows that the map $\psi:\Z/p\Z\rightarrow [0,p-1]\subseteq \Z$ defined by $\psi(sg)=s$  for $s\in [0,p-1]$, gives a Freiman isomorphism of $A'+B'$ with the integer sumset $\psi(-a_0+A')+\psi(-b_0+B')\subseteq \Z$. Observe that $${\gcd}^*(\psi(-a_0+A')+\psi(-b_0+B'))=1,$$ since if $\psi(-a_0+A')+\psi(-b_0+B')$ were contained in an arithmetic progression with difference $d\geq 2$, then this would also be the case for $\psi(-a_0+A')$ and $\psi(-b_0+B')$, and then $\ell_{dg}(A')+\ell_{dg}(B')<\ell_g(A')+\ell_g(B')$ would follow in view of $|A'|\geq |B'|\geq 2$, contradicting the minimality of $\ell_g(A')+\ell_g(B')$ for $g$.

In view of \eqref{rectify-large-bound-sym} and  $|B'|\leq |A'|$, we have $|A'+B'|\leq |2A|\leq |A'|+|B'|+\min\{|A'|,\,|B'|\}-4$. Thus, since  $\gcd^*(\psi(-a_0+A')+\psi(-b_0+B'))=1$, we can apply the $3k-4$ Theorem (Theorem \ref{thm-3k4Z}) to the isomorphic sumset $\psi(-a_0+A')+\psi(-b_0+B')$. Then, letting $P_{A}=a_0+[0,m]_g$, \ $P_{B}=b_0+[0,n]_g$ and $P_{A+B}\subseteq A'+B'$ be the resulting arithmetic progressions with common difference $g$, we conclude that \be |P_A\setminus A'|\leq r'+1\quad \und  \quad |P_B\setminus B'|\leq r'+1.\label{rolling-a-sym}\ee  If $A'=A$ and $B'=A$, then the original sumset $2A$ is rectifiable, we have $r'=r$, and the theorem follows with $P_A=P_B$ and $P_{2A}=P_{A+B}$ as just defined. Therefore we can assume otherwise, which in view of $|B'|\leq |A'|$ means  \be\label{AorBnontempty-sym}\quad A\setminus B'\neq \emptyset.\ee

Let $\Delta=|2A|- |A'+B'|\geq 0$. Then
 \be\label{r'-mainbound-sym} r'= |A\setminus A'|
 +|A\setminus B'|+r-\Delta.\ee
Since $|A'|+|B'|+r'=|A'+B'|=|2A|-\Delta$,  it follows
 from
 \eqref{rectify-large-bound-sym} and $|B'|\leq |A'|$ that
 \be\label{1st-r'} r'\leq |B'|-4-\Delta\quad\und\quad r'\leq |A'|-4-\Delta.\ee
 Averaging both bounds in \eqref{1st-r'}, using \eqref{r'-mainbound-sym}, and recalling that $|2A|=2|A|+r$, we obtain  \be\label{r'-special-sym} r'\leq \frac13|2A|-\frac83-\Delta.\ee

\subsection*{Step A} $|-A'+\overline{A'+A}|\leq |\overline{A'+A}|+2|A'|-4$.

\begin{proof}
If Step A fails, then combining its failure with $p-|2A|=|\overline{2A}|\leq |\overline{A'+A}|$ and Lemma \ref{lemma-dual} yields  $$p-|2A|+2|A'|-3\leq |\overline{A'+A}|+2|A'|-3\leq |-A'+\overline{A'+A}|\leq |\overline {A}|=p-|A|,$$
which implies that $|A|+2|A'|-3\leq |2A|$. This together with \eqref{rectify-large-bound-sym} and $|B'|\leq |A'|\leq |A|$ implies
$|A|+2|A'|-3\leq |A'|+2|B'|-4\leq |A|+2|A'|-4$, which is not possible.
\end{proof}

\subsection*{Step B} $|-A'+\overline{A'+A}|\leq |A'|+2|\overline{A'+A}|-3$.

\begin{proof}
If Step B fails, then combining its failure with $2p-4|A|-2r=2|\overline{2A}|\leq 2|\overline{A'+A}|$ and Lemma \ref{lemma-dual} yields  $$|A'|+2p-4|A|-2r-2\leq |A'|+2|\overline{A'+A}|-2\leq |-A'+\overline{A'+A}|\leq |\overline {A}|=p-|A|.$$ Collecting terms in the above inequality, multiplying by $2$, and applying the estimates $|B'|\leq |A'|$ and \eqref{r'-special-sym} yields \begin{align*}2p&\leq 6|A|+4r-2|A'|+4\leq 3|2A|+r-|A'|-|B'|+4\\&=3|2A|-|A'+B'|+r+r'+4=2|2A|+\Delta+r+r'+4
\; \leq \frac{7}{3}|2A|+r+\frac43.\end{align*}
Hence $|2A|\geq \frac67p-\frac37 r-\frac47$. Combined with \eqref{onetwo} and \eqref{betabound}, we conclude that
\[
\frac67p-\frac37\alpha\Big(\frac{3p+5}{8}\Big)+\frac57
<\frac67p-\frac37\beta|A|+\frac57=\frac67p-\frac37 r-\frac47\leq |2A|\leq \varepsilon p\leq \frac34 p,
\]
which yields the contradiction $0<(\frac67-\frac34-\frac{9}{56}\alpha)p
<\frac{15}{56}\alpha-\frac57<0$ (in view of $\alpha<\frac12$), completing Step B.
\end{proof}

By our application of the $3k-4$ Theorem (Theorem \ref{thm-3k4Z}) to $\psi(-a_0+A')+\psi(-b_0+B')$, we know that $A'+B'$ contains an arithmetic progression $P_{A+B}$ with difference $g$ and length $|P_{A+B}|\geq |A'|+|B'|-1$, which implies $$\ell_g(\overline{A'+B'})\leq p-|A'|-|B'|+1.$$
By \eqref{rolling-a-sym} and  \eqref{1st-r'}, we obtain \be\label{A'length}\ell_g(-A')=\ell_g(A')\leq |A'|+r'+1\leq |A'|+|B'|-3,\ee whence $\ell_g(-A')+\ell_g(\overline{A'+B'})\leq p-2$, ensuring $-A'+\overline{A'+B'}$ is rectifiable via the difference $g$. Since $\overline{A'+A}\subseteq \overline{A'+B'}$, it follows that $-A'+\overline{A'+A}$ is also rectifiable via the difference $g$.

By our application of the $3k-4$ Theorem (Theorem \ref{thm-3k4Z}) to $\psi(-a_0+A')+\psi(-b_0+B')$, we know $\psi(-a_0+A')$ is contained in the arithmetic progression $\psi(-a_0+P_A)=[0,m]$  with difference $1$ and length $|P_A|\leq |A'|+r'+1$, with the latter inequality by \eqref{rolling-a-sym}. Moreover, $r'+1\leq |B'|-3\leq |A'|-3$ (by \eqref{1st-r'}), so that $|A'|> \lceil\frac12|P_A|\rceil$, meaning $\psi(-a_0+A')$ must contain at least $2$ consecutive elements. Hence \be\label{gcd-a}{\gcd}^*(\psi(-a_0+A'))=1.\ee

Since $-A'+\overline{A'+A}$ is rectifiable via the difference $g$, it is then  isomorphic to the integer sumset $\psi(a_0+mg-A')+\psi(x+\overline{A'+A})$ for an appropriate $x\in \Z/p\Z$. Hence, in view of  \eqref{gcd-a}, Step A and Step B, we can apply the $3k-4$ Theorem (Theorem \ref{thm-3k4Z}) to the isomorphic sumset $\psi(a_0+mg-A')+\psi(x+\overline{A'+A})$ and thereby conclude that there is an arithmetic progression $P\subseteq -A'+\overline{A'+A}$ with difference $g$ and length $|P|\geq |A'|+|\overline{A'+A}|-1\geq |A'|+|\overline{2A}|-1=p-|2A|+|A'|-1$. Consequently, since Lemma \ref{lemma-dual} ensures that $P\subseteq -A'+\overline{A'+A}\subseteq \overline{A}$, it follows that $\ell_g(A)\leq |2A|-|A'|+1$. Combined with \eqref{A'length}, we find that
\be\label{rectifygo}\ell_g(A')+\ell_g(A)\leq |2A|+r'+2.\ee

If $A'+A$ is not rectifiable, then $\ell_g(A')+\ell_g(A)\geq p+2$, hence by \eqref{r'-special-sym} and \eqref{rectifygo} we have $p\leq |2A|+r'\leq \frac43|2A|-\frac83$, whence $|2A|\geq \frac34p+2>\varepsilon p$, contrary to hypothesis. Therefore  $A'+A$ is rectifiable. This contradicts the maximality of $|A'|+|B'|$ since by \eqref{AorBnontempty-sym} we have $|A|>|B'|$, which  completes Case 1.

\medskip\noindent\textbf{Case 2:} Every  pair of subsets $A'\subseteq A$  and $B'\subseteq A$ with $|B'|\leq |A'|$ whose sumset $A'+B'$ is rectifiable has \be\label{rectify-smalll-bound} |A'|+2|B'|\leq |2A|+3.\ee

Let $\ell:=|2A|=2|A|+r$. For the rest of this proof, let us identify $\Z/p\Z$ with the set of integers $[0,p-1]$ with addition mod $p$. Then, for every $X\subseteq \Z/p\Z$ and $d\in \Z/p\Z$, we define the exponential sum $S_X(d)=\sum_{x\in X}e^{\frac{2\pi i}{p}dx} \in \C$.

The idea is to use Freiman's estimate \cite[Theorem 1]{Lev} for such sums to show that the assumption \eqref{rectify-smalll-bound} implies
\be
\label{A-B-bound} |S_A(d)|\leq \tfrac{1}{3}|A|+\tfrac{2}{3} r+2\quad \mbox{for all nonzero }\; d\in \Z/p\Z.
\ee

For any $u\in [0,2\pi)$, consider the open arc $C_u=\{e^{i x}:\; x\in (u,u+\pi)\}$ of length $\pi$ in the unit circle in $\C$. Let $A'=\{x\in A:e^{\frac{2\pi i}{p}d x}\in C_u\}$. Since the set of $p$-th roots of unity contained in $C_u$ correspond to an arithmetic progression of difference 1 in $\Z/p\Z$, it is clear that, for $d^*$ the multiplicative inverse of $d$ modulo $p$, we have $\ell_{d^*}(A')\leq \frac{p+1}{2}$. Hence the sumset $A'+A'$ is rectifiable.
 Then the assumption \eqref{rectify-smalll-bound} implies that $3|A'|\leq |2A|+3$. This shows that every open half arc of the unit circle contains at most $n=\frac13 |2A|+1$ of the $|A|$ terms involved in the sum $S_A(d)$. By \cite[Theorem 1]{Lev} applied with this $n$, $N=|A|$, and $\varphi=\pi$, we obtain $|S_A(d)|\leq 2n-N = \frac23 |2A|+2 - |A|$, and \eqref{A-B-bound} follows.

To complete the proof, we now exploit \eqref{A-B-bound} to obtain a contradiction, using in particular the following manipulations which are standard in the additive combinatorial use of Fourier analysis (e.g. \cite[pp. 290--291]{Gr})

By Fourier inversion and the fact that $S_A(0)=|A|$ and $S_{2A}(0)=\ell$, we have
\ber
|A|^2p & = & \Summ{x\in \Z/p\Z}S_A(x)S_A(x)\overline{S_{2A}(x)}
\; = \; S_A(0)S_A(0)\overline{S_{2A}(0)}+\Summ{x\in (\Z/p\Z)\setminus \{0\}}S_A(x)S_A(x)\overline{S_{2A}(x)}\nn \\
&=& |A|^2\ell+\Summ{x\in (\Z/p\Z)\setminus \{0\}}S_A(x)S_A(x)\overline{S_{2A}(x)}\; \leq \;  |A|^2\ell+\Summ{x\in (\Z/p\Z)\setminus\{0\}}|S_A(x)||S_A(x)||S_{2A}(x)| \nn \\
&\leq & |A|^2\ell+ (\tfrac13|A|+\tfrac23 r+2) \Summ{x\in (\Z/p\Z)\setminus\{0\}}|S_A(x)||S_{2A}(x)|\nn.
\eer
This last sum is at most
 $\Big(\Summ{x\in \Z/p\Z\setminus\{0\}}|S_A(x)|^2\Big)^{1/2}\Big(\Summ{x\in \Z/p\Z\setminus \{0\}}|S_{2A}(x)|^2\Big)^{1/2}$ by the Cauchy-Schwarz inequality. We thus conclude that
 \[
|A|^2p \leq |A|^2\ell+\frac{|A|+2 r+6}{3} (|A|p-|A|^2)^{1/2}(\ell p-\ell^2)^{1/2}.
\]
Rearranging this inequality, we obtain
 \be \label{treeroot}
\frac{|A|+2 r+6}{3|A|}\geq
\frac{|A|(p-\ell)}{|A|^{1/2}(p-|A|)^{1/2}\ell^{1/2}(p-\ell)^{1/2}}=\left( \frac{\frac{p}{\ell}-1}{\frac{p}{|A|}-1}\right)^{1/2}.
\ee
By hypothesis $r= \beta|A|-3$, and $\ell=|2A|=(2+\beta)|A|-3$, so $|A|= \frac{\ell+3}{2+\beta}>\frac{\ell}{2+\beta}$. Using these estimates in  \eqref{treeroot} yields
\be
\nn \frac{1+2\beta}{3}= \frac{|A|+2(\beta|A|-3)+6}{3|A|} =  \frac{|A|+ 2 r+6}{3|A|}\geq
\left( \frac{\frac{p}{\ell}-1}{\frac{p}{|A|}-1}\right)^{1/2}> \left( \frac{\frac{p}{\ell}-1}{(2+\beta)\frac{p}{\ell}-1}\right)^{1/2}.
\ee
Rearranging the above inequality yields (in view of $0<\beta\leq \alpha<1$)
 \ber\label{rhs} \varepsilon p\geq \ell>\frac{1-(\frac{1+2\beta}{3})^2(2+\beta)}{1-(\frac{1+2\beta}{3})^2}p.
 \eer
Since $\beta\leq \alpha<1$, rearranging the above inequality yields
\be\label{contra}
4\beta^3+(12-4\varepsilon)\beta^2+(9-4\varepsilon)\beta+8\varepsilon -7>0.\ee  Thus $f(\beta)>0$, with $f(x)=4x^3+(12-4\varepsilon)x^2+(9-4\varepsilon)x+8\varepsilon -7$. As noted at the start of the proof, $f(x)$ is increasing for $x\geq 0$ with a unique positive root $\alpha$. As a result, \eqref{contra} ensures that $\beta>\alpha$, which is contrary to hypothesis, completing the proof.
\end{proof}

\begin{remark}
\textup{Our restriction $|2A|\leq \frac34 p$ in Theorem \ref{fouier-nightmare-calc} could be relaxed somewhat further, but at increasingly greater cost to the resulting constant $\alpha$. One simply needs to strengthen the hypothesis of \eqref{rectify-large-bound-sym} and appropriately adjust the Fourier analytic calculation in Case 2 in the above proof, using the correspondingly weakened inequality for \eqref{rectify-smalll-bound}.}
\end{remark}

\begin{proof}[Proof of Theorem \ref{thm:main1}]
As mentioned earlier, Theorem \ref{thm:main1} is just the special case of Theorem \ref{fouier-nightmare-calc} with $\varepsilon=\frac34$.
\end{proof}

We now proceed to prove the variant that we shall apply in the next section.

\begin{proof}[Proof of Theorem \ref{thm:main2}]
The proof is very close to that of Theorem \ref{fouier-nightmare-calc}, with the most significant difference occurring in Case 2. We only highlight the few differences in the argument.

First observe that, if $p=2$, then $|2A|<p$ forces $|A|=1$, in which case the theorem holds trivially. Therefore we can assume $p\geq 3$.
 Next observe (via Taylor series expansion) that $p\sin(\pi/p)$ is an increasing function for $p>1$ with limit $\pi$. The function $\eta/\sin(\pi\eta/3)$ is also an increasing function for $\eta\in (0,1)$. Thus
$\alpha\leq-\frac54+\frac14\sqrt{9+8\pi/\sin(\pi/3)}<0.3$. By hypothesis, $|A|\leq \frac{p-r}{3}=\frac13p-\frac13\beta|A|+1$, implying \be\label{beta} |A|\leq \frac{p+3}{\beta+3}<\frac{p+3}{3},\ee which replaces \eqref{betabound} for the proof. Also, $|2A|=2|A|+r\leq 2(\frac{p-r}{3})+r=\frac{2p+r}{3}$.

At the end of Step B in Case 1, we instead obtain $\frac67p-\frac37 r-\frac47\leq |2A|\leq \frac{2p+r}{3}$, which implies
\[
\frac23p\geq \frac67p-\frac{16}{21}r-\frac47\geq  \frac67p-\frac{16}{21}\alpha|A|+\frac{16}{7}-\frac47> \frac67p-\frac{16}{21}\alpha\Big(\frac{p+3}{3}\Big)+\frac{16}{7}-\frac47,
\]
with the final inequality above in view of \eqref{beta}. Thus $0<(\frac67-\frac23-\frac{16}{63}\alpha)p<\frac{16}{21}\alpha-\frac{12}{7}<0$ (in view of $0<\alpha<0.3$), which is the contradiction that instead completes Step B.

At the end of Case 1, we instead likewise obtain
\[
\frac34p+2\leq |2A|\leq \frac{2p+r}{3}\leq \frac23 p+\frac13\alpha|A|-1<\frac23p+\frac13\alpha\Big(\frac{p+3}{3}\Big)-1.
\]
This yields the contradiction $0<(\frac34-\frac23-\frac{\alpha}{9})p<\frac{\alpha}{3}-3<0$ (in view of $0<\alpha<0.3$) in order to complete Case 1.

\medskip

For Case 2, we begin by following the argument that proves \eqref{A-B-bound}, except that we use Lev's sharper estimate \cite[Theorem 2]{Lev} instead of \cite[Theorem 1]{Lev}. Thus, using that any two distinct terms in $S_A$ have the shortest arc between them of length at least $\delta=2\pi/p$, we obtain by \cite[Theorem 2]{Lev} applied with $n=\frac{1}{3}|2A|+1\leq p/2$ (so $\delta n\leq \pi$) that, for every such nonzero $d$, we have
\be\label{last}|S_A(d)|\leq \frac{\sin\Big(\big(\frac13 |2A|+1-\tfrac{1}{2}|A|\big)\tfrac{2\pi}{p}\Big)}{\sin(\frac{\pi}{p})}  = \frac{\sin\big((\frac13|A|+\frac23 r+2)\tfrac{\pi}{p}\big)}{\sin(\frac{\pi}{p})}.
\ee
Let $M=\frac13|A|+\frac23 r+2$, and let $y=M/p$. Note $M\leq (\frac13+\frac23\alpha)|A|<(\frac13+\frac23(0.3))\frac{p+3}{3}<\frac{p}{2}$ in view of $r\leq \alpha|A|-3$ and \eqref{beta}, ensuring $y\in (\frac{\eta}{3},\frac{1}{2})$. Then the  inequality in \eqref{last} becomes $|S_A(d)|\leq\frac{\sin(y \pi)}{y p \sin(\frac{\pi}{p})}\,M$. The function $f(p,y)=\frac{\sin(y\pi)}{y p \sin(\frac{\pi}{p})}$  is decreasing in $y\in (0,1/2)$ for any fixed $p\geq 3$, as can be seen by considering the taylor series expansion of its partial derivative. It is also decreasing in $p$ for every fixed $y\in (0,1/2)$ by a similar analysis. Letting $\gamma= f(p,\frac{\eta}{3})>0$, we can therefore replace \eqref{A-B-bound} by the  bound
\be\label{A-B-bound2}
|S_A(d)|\leq \gamma (\tfrac{1}{3}|A|+\tfrac{2}{3}r+2).
\ee
Since $M\frac{\pi}{p}<\frac{\pi}{2}$, $M>1$ and $p\geq 3$, it follows that $\sin(M\frac{\pi}{p})-M\sin(\frac{\pi}{p})\leq 0$ (as can be seen by considering derivatives with respect to $M$ and using the Taylor series expansion of $\tan(\frac{\pi}{p})$ to note $\tan(\frac{\pi}{p})>\frac{\pi}{p}$). Consequently, we see that the bound in \eqref{last} is at most $M$, ensuring  $\gamma\leq 1$.
We now obtain the following inequality instead of \eqref{treeroot}:
\be \label{treeroot2}
\gamma\frac{1+2\beta}{3}= \frac{\gamma (\frac13|A|+\frac23 r+2)}{|A|}\geq
\frac{|A|(p-\ell)}{|A|^{1/2}(p-|A|)^{1/2}\ell^{1/2}(p-\ell)^{1/2}}=\left( \frac{\frac{p}{\ell}-1}{\frac{p}{|A|}-1}\right)^{1/2}.
\ee
A similar rearrangement as the one that yielded \eqref{rhs} now leads to
\be\label{rhs-ii}\frac{2p+\frac{\beta}{3+\beta}(p+3)-3}{3}\geq \frac{2p+\beta|A|-3}{3}=\frac{2p+r}{3}\geq |2A|>\frac{1-\gamma^2(\frac{1+2\beta}{3})^2(2+\beta)}{1-\gamma^2(\frac{1+2\beta}{3})^2}p,
\ee
with the first inequality following from \eqref{beta}.
Since $0\leq \beta<1$ and $0<\gamma\leq 1$, we have $\frac{\beta}{3+\beta}<1$ and also $1-\gamma^2(\frac{1+2\beta}{3})^2>0$, so \eqref{rhs-ii} implies
$
\Big(\frac{\beta+2}{\beta+3}\Big)\Big(1-\gamma^2\Big(\frac{1+2\beta}{3}\Big)^2\Big)
>1-\gamma^2\Big(\frac{1+2\beta}{3}\Big)^2(2+\beta)$.
Multiplying both sides by $\beta+3>0$ and grouping on the left side the terms involving $\gamma$, we obtain  $(\beta+2)^2\gamma^2\Big(\frac{1+2\beta}{3}\Big)^2>1$. Taking square roots and expanding, we deduce $
2\beta^2+5\beta+2-3\gamma^{-1}>0$. The quadratic formula thus implies that either $\beta<\frac{-5-\sqrt{9+24\gamma^{-1}}}{4}<0$ or $\beta>\frac{-5+\sqrt{9+24\gamma^{-1}}}{4}=\alpha$. Since $\beta>0$, this contradicts the hypothesis $\beta\leq \alpha$, completing the proof.
\end{proof}

\section{Bounds for $m$-sum-free sets in $\Z/p\Z$}\label{sec:msf}

In this section, we give new bounds for the quantity $d_m(\Z/p\Z)$ defined in formula \eqref{eq:def-d_m} and for the associated limit
\[
d_m=\lim_{\substack{p\to\infty\\ p\textrm{ prime}}} d_m(\Z/p\Z).
\]
In Subsection \ref{subsec:lbs} we present some examples of large $m$-sum-free sets, and in Subsection \ref{subsec:ubs} we apply Theorem \ref{thm:main2} to give a new upper bound for $d_m(\Z/p\Z)$. In Subsection \ref{subsec:sumfree} we obtain an improvement of Theorem \ref{thm:D&L}.

\subsection{Lower bounds for $d_m(\Z/p\Z)$}\label{subsec:lbs}\hfill\\
As mentioned in the introduction, a simple example of a large $m$-sum-free set is the interval $(\frac{2}{m^2-4}p,\frac{m}{m^2-4}p)$, having asymptotic density  $\frac{1}{m+2}$ as $p\to\infty$. This gives the largest known size of $m$-sum-free sets for $m\leq 6$, but not for greater values of $m$. Indeed, there is the following construction, following an idea due to Tomasz Schoen. The version given below incorporates a suggestion of the anonymous referee that, with some additional modification, yielded the result as stated below. As noted in the proof, for fixed $m$, the constant $\frac18$ can be improved by a small factor tending to $0$ as $m\rightarrow \infty$.
\begin{lemma}\label{Schoen}
For each integer $m\geq 7$, we have $d_m(\Z/p\Z)\geq \frac18
(1-\frac{1}{p})$ for every prime $p$ of the form $p=4m^2n+1$. In particular, $d_m\geq \frac18$.
\end{lemma}

\begin{proof}
We identify $\Z/p\Z$ with the interval of integers $[0,p-1]$ with addition mod $p$. Let $\lambda\in [3,m]$ and $\mu\in [0,m-1]$ be integer parameters to be fixed later, and consider the interval
\[
J=\{4mn+1,4mn+2,\ldots,2\lambda mn\}\subset [0,p-1].
\]
We define an $m$-sum-free set $A$ by picking elements from $J$ in appropriate congruence classes mod $m$:
\[
A:=\{x\in J: x\!\!\mod m\in [0,\mu]\}.
\]
Note that the sumset $A+A$ taken in $\mb{Z}$ is a subset of $[0,p-1]$ because $\lambda\leq m$ guarantees $2\max A=4mn\lambda\leq 4m^2n= p-1$. We therefore have $y\in 2A$ $\Rightarrow$ $y\!\!\mod m\in [0,2\mu]$.

Now
\[
J=\bigcup_{i=1}^{\lceil\frac{\lambda}{2}\rceil-2}[(4i)mn+1,4(i+1)mn]\cup [(4(\lceil\tfrac{\lambda}{2}\rceil-1)mn+1,2\lambda mn]\subseteq \bigcup_{i=1}^{\lceil\frac{\lambda}{2}\rceil-1}[(4i)mn+1,4(i+1)mn].
\]
Since $p=4m^2n+1$, we have
$m\cdot [(4i)mn+1,4(i+1)mn]=\{m-i,2m-i,\ldots,4m^2n-i\}$ for $i\in [1,\lceil\frac{\lambda}{2}\rceil-1]$, meaning $m\cdot J$ is covered by the progressions
\[
U_i = \{m-i,2m-i,\ldots,4m^2n-i\},\quad i\in [1,\lceil\tfrac{\lambda}{2}\rceil-1].
\]
For $A$ to be $m$-sum-free it suffices to ensure that $2A\cap (m\cdot J)=\emptyset$, and for this it suffices for $\lambda$ and $\mu$ to satisfy $2\mu\leq m-(\lceil\frac{\lambda}{2}\rceil-1)-1$, that is,
\[
2\mu+\lceil\tfrac{\lambda}{2}\rceil \leq m.
\]
We also have $|A|=(\mu+1)\,\frac{|J|}{m}=2n(\mu+1)(\lambda-2)$, so
\[
\frac{|A|}{p-1} = \frac{(\mu+1)(\lambda-2)}{2m^2}.
\]

Suppose $m\equiv 0\mod 4$, so $m\geq 8$. Considering $\lambda=m$ and $\mu=\frac{m}{4}$ yields $\frac{|A|}{p-1} = \frac{(\mu+1)(\lambda-2)}{2m^2}=\frac{m^2+2m-8}{8m^2}$, which is at least $\frac{1}{8}$ for $m\geq 4$.

Suppose $m\equiv 1\mod 4$, so $m\geq 9$. Considering $\lambda=m$ and $\mu=\frac{m-1}{4}$ yields $\frac{|A|}{p-1} =\frac{(\mu+1)(\lambda-2)}{2m^2}=\frac{m^2+m-6}{8m^2}$, which is at least $\frac{1}{8}$ for $m\geq 6$. We remark that, taking $\lambda=m-3$ and $\mu=\frac{m+3}{4}$ instead yields $\frac{|A|}{p-1} =\frac{(\mu+1)(\lambda-2)}{2m^2}=\frac{m^2+2m-35}{8m^2}$, which is slightly better for larger values of $m$.

Suppose $m\equiv 2\mod 4$, so $m\geq 10$. In this case, we will modify the above construction with $\lambda=m-1$ and $\mu=\frac{m-2}{4}$. For these  parameters, we have $\lceil\frac{\lambda}{2}\rceil-1=\frac{m-2}{2}$,  \ $2(\mu+1)=\frac{m+2}{2}=m-\frac{m-2}{2}$, and
\begin{eqnarray}\label{tail}
m\cdot [(4(\lceil\tfrac{\lambda}{2}\rceil-1)mn+1,2\lambda mn] & = & m\cdot [2(m-2)mn+1,2(m-1)mn]\nonumber\\
& = & \{m-\tfrac{m-2}{2},2m-\tfrac{m-2}{2},\ldots,2m^2n-\tfrac{m-2}{2}\},
\end{eqnarray}
which is the subset of $m\cdot J\subseteq [0,p-1]$ congruent to $m-\frac{m-2}{2}$ modulo $m$.
Let $B=\{tm+\frac{m+2}{4}:\; mn\leq t\leq 2(m-1)n-1\}$ and set $A'=A\cup B$. Since $mn \leq t\leq 2(m-1)n-1$, we have $B\subseteq J$, while $2B=\{2m^2n+m-\frac{m-2}{2},2m^2n+2m-\frac{m-2}{2},\ldots,4(m-1)mn-m-\frac{m-2}{2}\}$, which is disjoint from the set in \eqref{tail}. Since $A+B$ is also disjoint from $m\cdot J$, it follows that $2A'\cap (m\cdot J)=\emptyset$, so $A'$ is $m$-sum-free. We have $\frac{|A'|}{p-1}=\frac{|A|+|B|}{p-1}=\frac{2n(\mu+1)(\lambda-2)+(m-2)n}{p-1}
=\frac{m^2+m-10}{8m^2}$, which is at least $\frac18$ for $m\geq 10$. We remark that, taking $\lambda=m-2$ and $\mu=\frac{m+2}{4}$ in the original construction instead yields $\frac{|A|}{p-1} =\frac{(\mu+1)(\lambda-2)}{2m^2}=\frac{m^2+2m-24}{8m^2}$, which is slightly better for larger values of $m$.

Suppose $m\equiv 3\mod 4$, so $m\geq 7$.
We  modify the original construction with $\lambda=m$ and $\mu=\frac{m-3}{4}$. For these  parameters, we have $\lceil\frac{\lambda}{2}\rceil-1=\frac{m-1}{2}$,  \ $2(\mu+1)=\frac{m+1}{2}=m-\frac{m-1}{2}$, and  \begin{align*}m\cdot [(4(\lceil\tfrac{\lambda}{2}\rceil-1)mn+1,2\lambda mn]&=m\cdot [2(m-1)mn+1,2m^2n]
\\&=\{m-\frac{m-1}{2},2m-\frac{m-1}{2},\ldots,2m^2n-\frac{m-1}{2}\},\end{align*} which is the subset of $m\cdot J\subseteq [0,p-1]$ congruent to $m-\frac{m-1}{2}$ modulo $m$.
Let $B=\{tm+\frac{m+1}{4}:\; mn\leq t\leq 2mn-1\}$ and set $A'=A\cup B$. Since $mn \leq t\leq 2mn-1$, we have $B\subseteq J$, while $2B=\{2m^2n+m-\frac{m-1}{2},2m^2n+2m-\frac{m-1}{2},\ldots,4m^2n-m-\frac{m-1}{2}\}$. Thus, similarly to the previous case, $2A'\cap (m\cdot J)=\emptyset$, so $A'$ is $m$-sum-free. We have $\frac{|A'|}{p-1}=\frac{|A|+|B|}{p-1}=\frac{2n(\mu+1)(\lambda-2)+mn}{p-1}
=\frac{m^2+m-2}{8m^2}$, which is at least $\frac18$ for $m\geq 2$. We remark that, taking $\lambda=m-1$ and $\mu=\frac{m+1}{4}$ in the original construction instead yields $\frac{|A|}{p-1} =\frac{(\mu+1)(\lambda-2)}{2m^2}=\frac{m^2+2m-15}{8m^2}$, which is slightly better for larger $m$.

In all four cases above, we obtain a set $A$ such that $d_m(\Z/p\Z)\geq \frac{|A|}{p}\geq \frac{|A|}{p-1}(1-\frac{1}{p})\geq \frac18(1-\frac{1}{p})$, and now the claim about the limit follows from the fact that by Dirichlet's Theorem there exist infinitely many primes in the arithmetic progression $\{4m^2n+1:n\geq 1\}$.
\end{proof}

\subsection{Upper bound for $d_m(\Z/p\Z)$}\label{subsec:ubs}\hfill\\
In this subsection we prove Theorem \ref{thm:dm-ub1}, which we restate here for convenience.
\begin{theorem}\label{thm:dm-ub2}
Let $p\geq 80$ be a prime, let $m$ be an integer in $[2,p-2]$, and let $c=c(p)$ be the solution  to the equation $c= \frac{1+3/p}{3+\alpha(c,p)}$, where $\alpha=\alpha(c,p)$ is the parameter in Theorem \ref{thm:main2} with $\eta=c$. Then $d_m(\Z/p\Z)<c$. In particular, $d_m\leq \frac{1}{3.1955}$.
\end{theorem}

The idea of the proof is roughly the following: either an $m$-sum-free set $A$ has doubling constant at least $2+\alpha$, in which case, since $(m\cdot A)\cap 2A=\emptyset$, we have $(3+\alpha)|A|\leq |(m\cdot A)|+|2A|\leq p$ and we are done, or we can apply Theorem \ref{thm:main2}, and thus, working with the two arithmetic progressions provided by the theorem, we reduce the problem essentially to bounding the size that two progressions $I$ and $J$ of equal difference can have if the dilate $m\cdot J$ has small intersection with $I$. Let us begin by establishing this result about progressions.

\begin{lemma}\label{AP-inter}
Let $p\ge 80$ be prime, let $0<\alpha\leq 1/5$, let $d\in [2,p-2]$, and let $N\in\mb{N}$. Let $I$ and $J$ be progressions in $\Z/p\Z$ having the same difference and satisfying $|I|=2N-1$, $|J|>(1+\alpha)N-3$, and $|I\cap (d\cdot J)| \le \alpha N-2$. Then $N<\frac{p+3}{3+\alpha}$.
\end{lemma}

\begin{proof}
First note that, without loss of generality, we can assume $d\le \frac{p-1}{2}$, since if the lemma is proved with this assumption, then, given $d> \frac{p-1}{2}$, we can multiply by $-1$ and apply the lemma with the intervals $-I$ and $J$. Let us proceed by contradiction supposing that there exists some $N$ (along with $p$, $d$, $\alpha$, $I$ and $J$) such that the hypotheses of the lemma are satisfied but $N\ge \frac{p+3}{3+\alpha}$. Note that the supposed properties of $I$ and $J$ are conserved if we dilate by the inverse of their difference mod $p$ and if we translate, replacing $I$ by $I+dz$ and $J$ by $J+z$. It follows that, identifying $\Z/p\Z$ with the integers $[0,p-1]$ with addition mod $p$, we can assume that $I=[p-|I|,p-1]$ and $J=x+[0,|J|-1]\mod p$ for some $x\in [0,p-1]$.

We claim that we can assume without loss of generality that
\be\label{alignhyp}
d\cdot x\in [0,d-1]\mod p.
\ee
Indeed, if this does not hold then either $d\cdot x\in [d, p-|I|+d-1]\mod p$, or $d\cdot x\in [p-|I|,p-1]$. If the former holds, then $d\cdot(x-1)\notin I\mod p$, so the interval $J'=(x-1)+[0,|J|-1]$ satisfies the hypotheses with $|I\cap (d\cdot J')|\leq |I\cap (d\cdot J)|$. On the other hand, if $d\cdot x\in [p-|I|,p-1]$, then letting $J'=(x+1)+[0,|J|-1]$ we have $d\cdot x\in I\cap (d\cdot J)$ and $d\cdot x\not\in I\cap (d\cdot J')$, so this interval $J'$ satisfies the hypotheses with $|I\cap (d\cdot J')|\leq |I\cap (d\cdot J)|$. In either case, by repeatedly shifting the interval $J$, we eventually obtain \eqref{alignhyp}.

Given \eqref{alignhyp}, we may partition  $d\cdot J$ into successive progressions $U_i$ (with difference $d$) for $i\in [1,s+1]$, such that $U_i=(\min U_i+d\Z)\cap [0,p-1]$  with $\min U_i\in [0,d-1]$ for $i\in [1,s]$, and $U_{s+1}$ is either empty or consists of an initial portion of $(\min U_{s+1}+d\Z)\cap [0,p-1]$ with $\min U_{s+1}\in [0,d-1]$.
Then $|U_i\cap I|\geq \left\lfloor\frac{|I|}{d}\right\rfloor$ for $i\in [1,s]$. It follows that $|(d\cdot J) \cap I|\ge s\left\lfloor\frac{|I|}{d}\right\rfloor$, whence
\begin{equation}\label{eqn:intervals-lemma-main}
\alpha N-2 \ge s\left\lfloor\frac{|I|}{d}\right\rfloor.
\end{equation}
Now, as $d\cdot x\in[0,d-1] \mod p$, each $U_i$ with $i\leq s$ starts in $[0,d-1]$ and ends in $[p-d,p-1]$, so $s$ is at least the number of consecutive intervals of length $p$ that fit inside $[0,|J|d-1]$:
\begin{equation}\label{eqn:s-bound}
s \geq \left\lfloor\frac{|J|d}{p}\right\rfloor > \frac{((1+\alpha)N-3)d}{p}-1.
\end{equation}
\noindent Substituting this lower bound for $s$ in \eqref{eqn:intervals-lemma-main}, as well as the bound $\left\lfloor \frac{|I|}{d}\right\rfloor \ge \frac{|I|}{d}-\frac{d-1}{d} = \frac{2N}{d}-1$, and expanding the resulting product, we obtain $\alpha N-2 > \frac{2(1+\alpha)}{p}N^2-\left( \frac{(1+\alpha)d}{p}+\frac{6}{p}+\frac{2}{d} \right)N+1+\frac{3d}{p}$. We group all terms involving $N$ on the right side, we note that the other terms grouped on the left side amount to a negative number, and we multiply through by $\frac{p}{2(1+\alpha)N}$, to deduce that
\begin{equation}\label{eqn:N-bound}
N < \frac{1}{2(1+\alpha)}\left( d(1+\alpha)+6+\frac{2p}{d}+\alpha p\right).
\end{equation}
We want to obtain a contradiction from this, using that $N\ge \frac{p+3}{3+\alpha}$. To that end, using the bounds $2\le d \le \frac{p-1}{2}$ on the right hand side of \eqref{eqn:N-bound} is not enough. However, we shall now show that we can assume $11\le d < p/6$, which will be enough.

First, we claim that $s\geq 1$. Indeed, otherwise $|J|\leq |(d\cdot J) \cap I|+ |(d\cdot J) \cap [0,p-|I|-1]| \le \alpha N-2+\left\lceil\frac{p-|I|}{d}\right\rceil$. Using the assumptions on $|I|, |J|$, and $d\ge 2$, we deduce  that $N< \frac{p+2d}{d+2}\leq \frac{p}{4}+2$. This, combined with our assumptions $N\ge(p+3)/(3+\alpha)$ and $\alpha<1/5$, contradicts $p\geq 80$.

Since $s\ge 1$, \eqref{eqn:intervals-lemma-main} yields $\alpha N -2  \ge \floor{|I|/d}\ge \frac{2N}{d}-1$. It follows  that $(\alpha N-1)d\geq2N >0$. Hence $\alpha N-1>0$ and $d\geq\frac{2N}{\alpha N-1}>\frac{2}{\alpha}$, whence  $d\geq 11$ follows in view of $\alpha\leq \frac{1}{5}$.

Note that $\floor{|I|/d}\ge 1$, for  otherwise $2N=|I|+1<d+1\leq \frac{p+1}{2}$, contradicting our assumptions $N\geq \frac{p+3}{3+\alpha}$ and $\alpha\leq 1/5$. Combining this with   \eqref{eqn:intervals-lemma-main} and \eqref{eqn:s-bound}, we obtain
$\alpha N-2 > \frac{((1+\alpha)N-3)d}{p}-1$, which means $d\le \left(\frac{\alpha N-1}{(1+\alpha)N-3}\right)p<\frac{\alpha}{1+\alpha} p$. As $\alpha\le 1/5$ we conclude that $d<p/6$.

Using now the bounds $11\le d < p/6$ in \eqref{eqn:N-bound}, and the assumption that $N\ge \frac{p+3}{3+\alpha}$, we deduce that $\frac{p+3}{3+\alpha} < \frac{p}{12}+\frac{p}{11(1+\alpha)}+\frac{\alpha p}{2(1+\alpha)}+\frac{3}{1+\alpha}$, implying $\frac{1}{3.2} < \frac{1}{12}+\frac{1}{11}+\frac{1}{10}+\frac{3}{p}$, contradicting $p\geq 80$.
\end{proof}

\begin{remark}
\textup{It is possible to extend the validity of Lemma \ref{AP-inter} to all primes $p\geq 5$, at the cost of lengthening the proof with several technicalities. The lemma has potential generalizations that seem of independent interest, though we do not need to pursue them for our purposes in this paper. For instance, the anonymous referee raised the question of which values of coefficients $\alpha,\beta$ and which functions $f(\alpha,\beta), g(\alpha,\beta) >0$ ensure that the following statement holds: if $I,J$ are arithmetic progressions in $\mb{Z}/p\mb{Z}$ with common difference and respective sizes $\alpha N +a$, $\beta N+b$, then $N> f(\alpha,\beta)p$   implies $|I\cap (d\cdot J)|> g(\alpha,\beta)N$.}
\end{remark}

We can now prove the main result.

\begin{proof}[Proof of Theorem \ref{thm:dm-ub2}]
Let $A\subseteq \Z/p\Z$ be an $m$-sum-free subset of maximum size, with $|A|=\eta p$, and let $\alpha=\alpha(\eta,p)=-\frac{5}{4} + \frac{1}{4}\sqrt{9+8\,\eta\, p\sin(\pi/p)/\sin(\pi\eta/3)}$. Assume by contradiction that $\eta\geq c$. Then, since $x\mapsto \frac{1+3/p}{3+\alpha(x,p)}$ is decreasing in $x\in (0,1)$ and $c=\frac{1+3/p}{3+\alpha(c,p)}$, we deduce that $\eta\geq c\geq  \frac{1+3/p}{3+\alpha}$, whence
\be\label{A-big}
|A|\geq \frac{p+3}{3+\alpha}>1.
\ee
As noted at the start of  the proof of Theorem \ref{thm:main2}, $\alpha(\eta,p)$ is increasing for $\eta\in (0,1)$ with $p\sin(\pi/p)\rightarrow \pi$ monotonically.
Since $2A$ and $m\cdot A$ are disjoint, we have $|2A|\leq p-|A|$, while $|2A|\geq 2|A|-1$ by the Cauchy-Davenport Theorem. Thus $2|A|-1\leq |2A|\leq p-|A|$, implying $|A|\leq \frac{p+1}{3}$ and $\eta\leq \frac{p+1}{3p}$. Since $p\geq 80$, we have $\eta\leq \frac{3}{8}$ and $\alpha\leq -\frac{5}{4}+\frac14\sqrt{9+3\pi/\sin(\pi/8)}<0.2$.

Let $|2A|=2|A|+r$. Since $A$ is $m$-sum-free, the sets $2A$ and $m\cdot A$ are disjoint, which implies that $|2A|<p$ (as $A$ is nonempty) and that $p\geq |2A|+|m\cdot A|=3|A|+r$. Thus $$|A|\leq \frac{p-r}{3}\quad\und\quad |2A|=2|A|+r\leq \frac{2p+r}{3}.$$
Since $|2A|<p$, the Cauchy-Davenport Theorem implies $r\geq -1$.

If $|2A|=2|A|+r>(2+\alpha)|A|-3$, then $r>\alpha|A|-3$, in which case $|A|\leq \frac{p-r}{3}<\frac{p-\alpha|A|+3}{3}$, which contradicts \eqref{A-big}. Therefore $|2A|\leq (2+\alpha)|A|-3$ and $r\leq \lfloor \alpha|A|-3\rfloor$. We can now apply Theorem \ref{thm:main2}. As a result, there are arithmetic progressions $P_A$ and $P_{2A}$ with common difference $g$ such that $A\subseteq P_A$, $P_{2A}\subseteq 2A$, $|P_A|= \lfloor(1+\alpha)|A|-2\rfloor\leq p$, and $|P_{2A}|= 2|A|-1$. It follows that $P:=m\cdot P_A$ is an arithmetic progression with difference $mg\neq \pm g$ such that
\[
|P\cap P_{2A}|\leq |P\cap 2A|\leq |P_A\setminus A|\leq \alpha|A|-2.
\]
We can therefore apply Lemma \ref{AP-inter} with $N=|A|$ (as $\alpha<0.2$), deducing that $|A|<\frac{p+3}{3+\alpha}$, a contradiction. Therefore we must have $\eta< c$, so $d_m(\Z/p\Z)<c$, which proves the first claim in the theorem. Taking the limit of $c$ as $p\to\infty$, we deduce that $d_m\leq t$, where $t$ is defined by the equation $t=F(t)$ for the function $F(t)=(\frac{7}{4} + \frac{1}{4}\sqrt{9+8\,t\, \pi/\sin(\pi t/3)})^{-1}$. Since $F$ is monotonically decreasing and satisfies $F(3.1955^{-1})<3.1955^{-1}$, we must have $t<3.1955^{-1}$, which proves the second claim in the theorem.
\end{proof}

\subsection{The structure of large sum-free sets in  $\Z/p\Z$}\label{subsec:sumfree}\hfill\smallskip\\
In this final part of the paper, we apply Theorem \ref{thm:main2} to obtain the following improvement of Theorem \ref{thm:D&L}.

\begin{theorem}
\label{thm-sumfree} Let $p\geq 14\, 000$ be prime and let $A\subseteq \Z/p\Z$ be sum-free with $|A|\geq (0.313)p$. Then $m\cdot A\subseteq [\,|A|,p-|A|\,]\subseteq \Z/p\Z$ for some $m\in [1,p-1]$.
\end{theorem}

\begin{proof}
By hypothesis, $|A|=\eta p>0$ with $\eta \geq 0.313$. Set $|2A|=2|A|+r$. Since $A$ is sum-free, we have $(2A)\cap A=\emptyset$, implying $2|A|+r=|2A|\leq p-|A|<p$, whence $|A|\leq \frac{p-r}{3}$.   As in the proof of Theorem \ref{thm:dm-ub2}, let $\alpha=\alpha(\eta,p)=-\frac{5}{4} + \frac{1}{4}\sqrt{9+8\,\eta\, p\sin(\pi/p)/\sin(\pi\eta/3)}$. Observe that $\alpha(\eta,p)$ is increasing as a function of $p\geq 2$ and $\eta\in (0,1)$, so $\alpha=\alpha(\eta,p)\geq \alpha(0.313,14000)\geq \beta:=0.195579$. If $|2A|> (2+\beta)|A|-3$, then we have $(2+\beta)|A|-3<|2A|\leq p-|A|$, implying $(0.313)p\leq |A|<\frac{p+3}{3+\beta}$, and thus $p\leq 13\, 875$, which is contrary to hypothesis. Therefore we instead conclude that $|2A|\leq(2+\beta)|A|-3\leq (2+\alpha)|A|-3$, allowing us to apply Theorem \ref{thm:main2} to conclude that there is an arithmetic progression $P\subseteq \Z/p\Z$ with $A\subseteq P$ and $|P|\leq |A|+r+1$. By dilating $A$ by the inverse of the difference of the progression $P$, we can assume without loss of generality that $P$ has difference $1$. Since $2|A|+r=|2A|\leq (2+\beta)|A|-3$, we have $r\leq \beta|A|-3$, and thus $|P|\leq |A|+r+1\leq (1+\beta)|A|-2$. The bound $|A|\leq (p+1)/3$ given by the Cauchy-Davenport theorem then implies $|P|\leq (1+\beta)(p+1)/3< \frac{p+1}{2}$. It follows that the sumset $A+A$ is rectifiable.

Let $\psi:A+A\rightarrow \Z$ be the associated Freiman isomorphism,  with coordinate map $\psi_A:A\rightarrow \Z$. Note that the map of the form $a_0+sg\mapsto s$ involved in the definition of $\psi_A$ (see the remarks before Lemma \ref{lemma-dual}) can be assumed to be just a translation (since the element $g$ here, being the difference of $P$, is assumed to be 1). By slight abuse of notation, we drop the subscript from $\psi_A$, denoting this map also by $\psi$. Let $\psi':A-A\rightarrow \Z$ be the Freiman isomorphism defined by $\psi'(x-y)=\psi(x)-\psi(y)$ for $x,\,y\in A$ (see the remarks after Theorem \ref{thm-3k4Z}). Since $|P|\leq |A|+r+1\leq 2|A|-2$ implies  $|A|>\frac{|P|+1}{2}$, we are assured that $A$ contains two consecutive elements in $P$, whence $\gcd^*(\psi(A))=1$. Since $A$ is sum-free, we have $(A-A)\cap A=\emptyset$, and thus $|A-A|\leq p-|A|$.
Since $A-A\cong \psi(A)-\psi(A)$, we have $|\psi(A)-\psi(A)|=|A-A|$ and $|\psi(A)|=|A|$. As a result, if $|\psi(A)-\psi(A)|\geq 3|\psi(A)|-3$, then $p-|A|\geq |A-A|=|\psi(A)-\psi(A)|\geq3|\psi(A)|-3=3|A|-3$, implying $(0.313)p\leq |A|\leq \frac{p+3}{4}$, contradicting that $p\geq 14\, 000$. Therefore $|\psi(A)-\psi(A)|\leq 3|\psi(A)|-4$, allowing us to apply the $3k-4$ Theorem (Theorem \ref{thm-3k4Z}) with the sets $\psi(A)$, $-\psi(A)$. This, together with the remarks in the paragraph above Theorem \ref{thm-3k4Z}, implies that $[-(|A|-1),(|A|-1)]\subseteq \psi(A)-\psi(A)$. Hence $[-(|A|-1),(|A|-1)]\subseteq \psi'(A-A)$, and given the form of $\psi'$ , it follows that in $\Z/p\Z$ we have $[-(|A|-1),(|A|-1)]\subseteq  A-A$. Since $A$ being sum-free implies $(A-A)\cap A=\emptyset$, this forces $A\cap [-(|A|-1),(|A|-1)]=\emptyset$, i.e., $A\subseteq [\, |A|,p-|A|\, ]$, which completes the proof.
\end{proof}

\noindent \textbf{Acknowledgements.} We are very grateful to Tomasz Schoen for providing the original idea of the construction in Lemma \ref{Schoen} and for useful remarks. We also thank very much the anonymous referee for insightful comments that helped us to improve this paper.


\begin{thebibliography}{99}
\bibitem{B&al} A. Baltz, P. Hegarty, J. Knape, U. Larsson, T. Schoen, \emph{The structure of maximum subsets of $\{1,\ldots,n\}$ with no solutions to $a+b=kc$}, Electron. J. Combin. \textbf{12} (2005), Paper No. R19, 16pp.

\bibitem{BHS} P.-Y. Bienvenu, F. Hennecart, I. Shkredov, \emph{A note on the set A(A+A)}. Mosc. J. Comb. Number Theory 8 (2019), no. 2, 179--188.

\bibitem{Bloom} T. F. Bloom, \emph{A quantitative improvement for Roth's theorem on arithmetic progressions}, J. Lond. Math. Soc. (2) \textbf{93} (2016), 643--663.

\bibitem{CSS} P. Candela, O. Serra, C. Spiegel, \emph{A step beyond Freiman's theorem for set addition modulo a prime},  J. Th\'eor. Nombres Bordeaux \textbf{32} (2020), no. 1, 275--289.

\bibitem{C&R} P. Candela, A. de Roton, \emph{On sets with small sumset in the circle}, Q. J. Math. \textbf{70} (2019), no. 1, 49--69.

\bibitem{C&S} P. Candela, O. Sisask, \emph{On the asymptotic maximal density of a set avoiding solutions to linear equations modulo a prime}, Acta Math. Hungar. \textbf{132} (2011), no. 3, 223--243.

\bibitem{D&L} J.-M. Deshouillers, V. F. Lev, \emph{A refined bound for sum-free sets in groups of prime order}, Bull. Lond. Math. Soc. \textbf{40} (2008), no. 5, 863--875.

\bibitem{C&G} F. R. K. Chung, J. L. Goldwasser, \emph{Integer sets containing no solutions to $x+y=3z$}, in: R.L. Graham and J. Nesetřil eds., The Mathematics of Paul Erd\H os, Springer, Berlin (1997), 218--227.

\bibitem{C&G-cts} F. R. K. Chung, J. L. Goldwasser, \emph{Maximum subsets of $(0,1]$ with no solutions to $x+y=kz$}, Electron. J. Combin. \textbf{3} (1996), no. 1, Research Paper 1.

\bibitem{Freiman61} G. Freiman, \emph{Inverse problems in additive number theory. Addition of sets of residues modulo a prime}, In Dokl. Akad. Nauk SSSR, volume 141, pages 571--573, 1961.

\bibitem{GR06} B. Green, I. Z. Ruzsa, \emph{Sets with small sumset and rectification}, Bull. Lond. Math. Soc. (1) \textbf{38} (2006), 43--52.

\bibitem{Gr} D. J. Grynkiewicz, \emph{Structural additive theory},
Developments in Mathematics, \textbf{30}. Springer, Cham, 2013. xii+426 pp.

\bibitem{Lev} V. F. Lev, \emph{Distribution of points on arcs}, Integers 5(2)(2005)(electronic)

\bibitem{LS} V. F. Lev, I. Shkredov, \emph{Small doubling in prime-order groups: from 2.4 to 2.6},  J. Number Theory \textbf{217} (2020), 278--291.

\bibitem{M&R} M. Matolcsi, I. Z. Ruzsa, \emph{Sets with no solutions to} $x + y = 3z$, European J. Combin. \textbf{34} (2013), no. 8, 1411--1414.

\bibitem{D&P}  A. Plagne, A. de Roton, \emph{Maximal sets with no solution to $x+y=3z$}, Combinatorica \textbf{36} (2016), no. 2, 229--248.

\bibitem{Rodseth06} \O. J. R{\o}dseth, \emph{On Freiman's 2.4-theorem}, Skr. K. Nor. Vidensk. Selsk, (4): 11--18, 2006.

\bibitem{Roth} K. F. Roth,  \emph{On certain sets of integers} J. London Math. Soc. \textbf{28} (1953), 104--109.

\bibitem{Sanders} T. Sanders, \emph{On Roth's theorem on progressions}, Ann. of Math. \textbf{174} (2011), 619--636.

\bibitem{SZ09} O. Serra, G. Z\'{e}mor, \emph{Large sets with small doubling modulo $p$ are well covered by an arithmetic progression},
Ann. Inst. Fourier (Grenoble) \textbf{59} (2009), no. 5, 2043--2060.

\bibitem{vosper} G. Vosper,
 \emph{The critical pairs of subsets of a group of prime order},
J. London Math. Soc.
\textbf{31} (1956),
200--205.
\end{thebibliography}
\end{document}